\documentclass[11 pt, final]{article}
\title{A counterexample to the easy direction of the geometric Gersten conjecture.}
\author{David Bruce Cohen}

\usepackage{fullpage}
\usepackage[latin1]{inputenc}
\usepackage{tikz}
\usetikzlibrary{shapes,arrows}
\usepackage{graphicx}
\usepackage{amsmath}
\usepackage{amssymb}
\usepackage{amsthm}
\usepackage{epsfig,pinlabel}
\usepackage{amsfonts}
\usepackage{amscd}

\DeclareMathOperator{\ZZ}{\ensuremath{\mathbb{Z}}}

\DeclareMathOperator{\CC}{\ensuremath{\mathbb{C}}}

\DeclareMathOperator{\HH}{\ensuremath{\mathbb{H}}}

\DeclareMathOperator{\Lip}{Lip}

\def\To{{\rightarrow}}

\def\la{{\langle}}
\def\ra{{\rangle}}

\theoremstyle{plain}
\newtheorem{theorem}{Theorem}[section]
\newtheorem{lemma}[theorem]{Lemma}

\theoremstyle{definition}
\newtheorem*{definition}{Definition}

\begin{document}

\maketitle

\begin{abstract}
For finitely generated groups $H$ and $G$, equipped with word metrics, a translation-like action of $H$ on $G$ is a free action such that each element of $H$ acts by a map which has finite distance from the identity map in the uniform metric. For example, if $H$ is a subgroup of $G$, then right translation by elements of $H$ yields a translation-like action of $H$ on $G$. Whyte asked whether a group having no translation-like action by a Baumslag-Solitar group must be hyperbolic, where the free abelian group of rank $2$ is understood to be a Baumslag-Solitar group. We show that the converse question has a negative answer, and in particular the fundamental group of a closed hyperbolic 3-manifold admits a translation-like action by the free abelian group of rank $2$.
\end{abstract}

\section{Introduction.}
A metric space $X$ is said to be uniformly discrete if it has a minimum distance, meaning
$$\inf\{d(x,y):x,y\in X; x\neq y\}>0,$$
and said to have bounded geometry if for all $r>0$, there is some $N_r>0$ such that every $r$-ball has cardinality at most $N_r$. If $X$ satisfies both of these conditions, it is said to be a UDBG space \cite[\S 2]{whyte}. For example, a finitely generated group equipped with a word metric is a UDBG space. More generally, if $X$ is the vertex set of a connected graph of bounded degree, equipped with the metric that assigns length $1$ to each edge, then $X$ is a UDBG space. In \cite{whyte}, Whyte introduced the following notion.

\begin{definition}\cite[Definition 6.1]{whyte}
Let $X$ be a UDBG space. A translation-like action of a group $H$ on $X$ is a free action by maps at a finite distance from the identity. That is, the action satisfies the following rules.

\begin{itemize}
\item For $x\in X$ and $h\in H$, if $h \cdot x= x$, then $h=1_H$.
\item For all $h\in H$, the set $\{d(x,h\cdot x):x\in X\}$ is bounded
\end{itemize}
\end{definition}

We will mostly be interested in the case where $H$ is finitely generated and the UDBG space $X$ is a finitely generated group $G$ equipped with a word metric. In this case, a translation-like action of $H$ on $G$ is just a vertex-surjective embedding of a disjoint union of copies of a Cayley graph of $H$ into a Cayley graph of $G$ (since an orbit of a translation-like action of $H$ on $G$ embeds the Cayley graph of $H$ into some Cayley graph of $G$).

\paragraph{Translation-like actions generalize subgroups.} If $H$ is a finitely generated subgroup of $G$, then $H$ acts translation-like on $G$ via
$$h\cdot g=gh^{-1}$$
for $h\in H$ and $g\in G$. Many properties which pass to subgroups of $G$ also pass to groups which act translation-like on $G$. For instance, Jeandel \cite[Theorem 3]{jeandel} has shown that if $G$ has no weakly aperiodic subshift of finite type, then the same is true for finitely presented groups acting translation-like on $G$. Whyte used this idea to give ``geometric" versions of several famous conjectures about how the properties of $G$ constrain its subgroups \cite[\S 6]{whyte}.

\paragraph{The Geometric Von Neumann-Day Conjecture.} The Von Neumann-Day conjecture (disproven by Olshanskii \cite{olshanskii}) asserts that a group $G$ should be nonamenable if and only if $G$ contains a free subgroup. Whyte used translation-like actions to formulate and prove a geometric version of this conjecture---namely, that $G$ is nonamenable if and only if $\ZZ\ast\ZZ$ acts translation-like on $G$ \cite[Theorem 6.1]{whyte}.

\paragraph{The Geometric Burnside Problem.} The Burnside problem (answered in the negative by Golod and Shafarevich \cite{gs}) asks whether every infinite finitely generated group contains a $\ZZ$-subgroup. The geometric Burnside problem asks whether every infinite, finitely generated group admits a translation-like action of $\ZZ$. Seward \cite[Theorem 1.4]{seward} proved that the answer to this question is yes.

\paragraph{The Geometric Gersten Conjecture.} Recall that for nonzero integers $m,n$, the Baumslag-Solitar group $BS(m,n)$ is the group presented by $\la a,b|ab^{m}a^{-1}=b^n\ra$, and in particular $BS(1,1)\cong\ZZ^2$. It is known that these groups do not embed in hyperbolic groups. The Gersten conjecture \cite[Q 1.1]{bestvina}---usually attributed to Gromov---roughly states that for a group satisfying some finiteness properties, hyperbolicity should be equivalent to having no Baumslag-Solitar subgroup. We do not know whether Gersten actually asked this question, although \cite{gersten} asks whether every finitely presented subgroup of a hyperbolic group must be hyperbolic. \cite{brady} showed that this was false, and hence that the Gersten conjecture is false for finitely presented groups (weaker versions remain open).

The geometric Gersten conjecture states being hyperbolic is equivalent to having no translation-like action by any $BS(m,n)$. In point of fact, Whyte only asked about the ``hard" direction---whether a group which is not hyperbolic must admit a translation-like action of a Baumslag-Solitar group---and only for 2-dimensional groups. By an observation of Jeandel\cite[\S 5]{jeandel}, knowing the hard direction for all amenable groups would imply that every group (except for virtually cyclic groups) has a weakly aperiodic subshift of finite type, as conjectured by Carroll and Penland\cite{cp}. In a recent preprint Jiang\cite{jiang} has shown that the lamplighter group admits no translation-like actions by Baumslag-Solitar groups. Since the lamplighter is not hyperbolic, this disproves the hard direction of the geometric Gersten conjecture, although finitely presented counterexamples remain unknown.

Seward \cite[\S 1.(3')]{seward} asked about the other direction---whether Baumslag-Solitar groups may act translation-like on hyperbolic groups. Our main theorem gives a negative answer to this question.

\begin{theorem}
\label{theorem:main}
Let $G$ be the fundamental group of a closed hyperbolic 3-manifold. Then $\ZZ^2$ acts translation-like on $G$.
\end{theorem}

\section{Proof of Theorem \ref{theorem:main}.}

Let $G$ be the fundamental group of a closed hyperbolic 3-manifold. We will prove Theorem \ref{theorem:main} by showing that $G$ is bilipschitz to a UDBG space which admits a translation-like action of $\ZZ^2$---the following lemma says that this is sufficient.

\begin{lemma}
\label{lemma:conjugation}
If $H$ acts translation-like on $X_1$, and $X_1$ is bilipschitz to $X_2$ then $H$ acts translation-like on $X_2$.
\end{lemma}

\begin{proof}
Let $\psi:X_1\To X_2$ be a bilipschitz map. We define a translation-like action of $H$ on $X_2$ by conjugating the action as follows. For $x\in X_2$, take
$$h\cdot x=\psi(h\cdot \psi^{-1}(x)).$$
It is clear that this is a free action, and it is translation-like because
$$d(x,h\cdot x)\leq \Lip(\psi)d(\psi^{-1}(x),h\cdot\psi^{-1}(x)).$$
\end{proof}

\begin{lemma}
\label{lemma:graph}
There exists a UDBG space $X$ such that $\ZZ^2$ acts translation-like on $X$ and $X$ is bilipschitz to $G$.
\end{lemma}

\begin{proof}

Consider the set of points
$$X=\{(2^c a,2^c b, 2^c):a,b,c\in\ZZ\}$$
in the upper half space model of $\HH^3$. The reader may verify that this is indeed a UDBG space (the shortest distance is $\log(2)$ and it is not hard to see that the size of $r$-balls in $X$ is roughly exponential in $r$).

To define a translation-like action of $\ZZ^2$ on $X$, let the generators $e_1,e_2$ of $\ZZ^2$ act by
$$e_1\cdot (2^c a,2^c b,2^c)=(2^c(a+1),2^c b,2^c)$$
and
$$e_2\cdot (2^c a,2^c b,2^c)=(2^c a,2^c (b+1),2^c).$$
These maps commute, each moves points by a distance of $1$, and the $\ZZ^2$-action they induce is clearly free, so it is translation-like.

Observe that $X$ is quasi-isometric to $\HH^3$ because every point of $\HH^3$ lies within a bounded distance of $X\subset\HH^3$. Thus, by the Svarc-Milnor theorem, $X$ is quasi isometric to $G$. By \cite[Theorem 2]{whyte}, any quasi-isometry between nonamenable UDBG spaces is at a bounded distance from a bilipschitz map, so $X$ is bilipschitz to $G$.
\end{proof}

Combining Lemma \ref{lemma:conjugation} and Lemma \ref{lemma:graph}, we have proved Theorem \ref{theorem:main}

\section{Questions}
We close with three questions.
\paragraph{Other Baumslag-Solitar groups.} Do any hyperbolic groups admit translation-like actions of Baumslag-Solitar groups $BS(m,n)$ with $m\geq 2$?

\paragraph{Other hyperbolic groups.} Which hyperbolic groups admit translation-like actions of $\ZZ^2$? Jiang\cite{jiang} has recently observed one may use results of \cite{bst} to show that $\ZZ^2$ cannot act translation-like on free groups, and it appears \cite[Proposition 4.1]{bst} that this technique may be used to rule out translation-like actions of $\ZZ^2$ on hyperbolic surface groups, but we have no idea whether such actions exist on hyperbolic one-relator groups or on random groups.

\paragraph{Gromov-Furstenberg for returns of the horospherical flow in a hyperbolic 3-manifold.} (See \cite{bk} for context). Let $\Gamma$ be a cocompact lattice in $PSL(2;\CC)$, let $H$ be an $\epsilon$-neighborhood of some horosphere $H_0$ in $\HH^3$, let $\ast\in\HH^3$, and consider the intersection $\mathcal{O}=(\Gamma\cdot\ast)\cap H$. If we equip $\mathcal{O}$ with the metric inherited from $H$, then $\mathcal{O}$ is quasi isometric to $H_0\cap(\Gamma\cdot B_\epsilon(\ast))$, where $B_\epsilon(\ast)$ denotes the $\epsilon$ ball around $\ast$ in $\HH^3$.  From Ratner's theorem \cite{ratner}, it then follows (with some thought) that $\mathcal{O}$ is quasi isometric to $\ZZ^2$. Must $\mathcal{O}$ be bilipschitz to $\ZZ^2$? This was our original attempt at finding a translation-like action of $\ZZ^2$ on $\Gamma$.

\section{Acknowledgments.} We wish to thank Kevin Whyte, Benson Farb, and Andy Putman for conversations. We wish to thank Yongle Jiang for sharing an early draft of his preprint with us. This work has been supported by NSF award 1502608.

\bibliographystyle{plain}
\bibliography{bibliography}

\begin{thebibliography}{10}

\bibitem{bst}
Itai Benjamini, Oded Schramm, and {\'A}d{\'a}m Tim{\'a}r.
\newblock On the separation profile of infinite graphs.
\newblock {\em Groups, Geometry, and Dynamics}, 6(4):639--658, 2012.

\bibitem{bestvina}
Mladen Bestvina.
\newblock Questions in geometric group theory.
\newblock Problem list, 2004.

\bibitem{brady}
Noel Brady.
\newblock Branched coverings of cubical complexes and subgroups of hyperbolic
  groups.
\newblock {\em Journal of the London Mathematical Society}, 60(02):461--480,
  1999.

\bibitem{bk}
Dmitri Burago and Bruce Kleiner.
\newblock Rectifying separated nets.
\newblock {\em Geometric and Functional Analysis}, 12(1):80--92, 2002.

\bibitem{cp}
David Carroll and Andrew Penland.
\newblock Periodic points on shifts of finite type and commensurability
  invariants of groups.
\newblock {\em New York J. Math}, 21:811--822, 2015.

\bibitem{gersten}
SM~Gersten.
\newblock Subgroups of word hyperbolic groups in dimension 2.
\newblock {\em Journal of the London Mathematical Society}, 54(2):261--283,
  1996.

\bibitem{gs}
Evgeniy~Solomonovich Golod.
\newblock On nil-algebras and finitely approximable p-groups.
\newblock {\em Izvestiya Rossiiskoi Akademii Nauk. Seriya Matematicheskaya},
  28(2):273--276, 1964.

\bibitem{jeandel}
Emmanuel Jeandel.
\newblock Translation-like actions and aperiodic subshifts on groups.
\newblock {\em arXiv preprint arXiv:1508.06419}, 2015.

\bibitem{jiang}
Yongle Jiang.
\newblock Translation-like actions yield regular maps.
\newblock {\em arXiv preprint arXiv:1703.09253}, 2017.

\bibitem{olshanskii}
Aleksandr~Yur'evich Ol'shanskii.
\newblock On the problem of the existence of an invariant mean on a group.
\newblock {\em Russian Mathematical Surveys}, 35(4):180--181, 1980.

\bibitem{ratner}
Marina Ratner.
\newblock On raghunathan's measure conjecture.
\newblock {\em Annals of Mathematics}, 134(3):545--607, 1991.

\bibitem{seward}
Brandon Seward.
\newblock Burnside's problem, spanning trees and tilings.
\newblock {\em Geometry \& Topology}, 18(1):179--210, 2014.

\bibitem{whyte}
Kevin Whyte.
\newblock Amenability, bilipschitz equivalence, and the von neumann conjecture.
\newblock {\em Duke mathematical journal}, 99(1):93--112, 1999.

\end{thebibliography}
\end{document}